\documentclass[reqno]{amsart}

\pdfoutput=1
\usepackage{amssymb,amsmath,latexsym}
\usepackage{empheq}

\newtheorem{theorem}{Theorem}
\newtheorem{lemma}[theorem]{Lemma}
\newtheorem{remark}[theorem]{Remark}
\newtheorem{proposition}[theorem]{Proposition}
\newtheorem*{theorem*}{Theorem}

\newtheorem{corollary}[theorem]{Corollary}
\newtheorem*{conjecture*}{Conjecture}

\newcommand{\diam}{\operatorname{diam}}

\title{Some asymptotic estimates on the Von Neumann Inequality for
homogeneous polynomials}

\author{Oscar Zatarain-Vera}
\address{Department of Mathematical Sciences, Kent State University, Kent, OH 44242, USA}
\email{ozatarai@kent.edu}

\begin{document}
    \maketitle

    \begin{abstract}
        We will discuss some results of the paper "Asymptotic estimates on the Von Neumann Inequality for homogeneous polynomials", of Galicer D., Muro S. and Sevilla-Peris P. \cite{Ga}. Also, we will see how to extend some of these results using similar techniques.
    \end{abstract}
    \vspace{1cm}

    \section{Preliminaries}

    First, let's recall some definitions and notation. We denote by $\mathbb{N}, \mathbb{N}_0, \mathbb{R}, \mathbb{C}, \mathbb{D}$ and $\overline{\mathbb{D}}$ the corresponding set of natural numbers, set of non-negative numbers, set of real numbers, set of complex numbers, set of complex numbers with modulus less than 1 and set of complex numbers with modulus less or equal to one. If $(a_n)$ and $(b_n)$ are two sequences of real numbers, we write $a_n\ll b_n$ if there is a positive constant $C$ (independent of $n$) such that $a_n\leq Cb_n$ for every $n$. We also write $a_n\sim b_n$ if $a_n\ll b_n$ and $b_n\ll a_n$. Given a set $A$ its cardinality is denoted by $|A|$.

    A $k-$homogeneous polynomial in $n$ variables is a function $p:\mathbb{C}^n\to \mathbb{C}$ of the form
    \[
        p(z_1,...,z_n)=\sum_{\substack{\alpha\in \mathbb{N}^n_0 \\ |\alpha|=k}}a_{\alpha}z_1^{\alpha_1}\cdots z_n^{\alpha_n}=\sum_{\substack{J=(j_1,...,j_k) \\ 1\leq j_1\leq ...\leq j_k\leq n}}c_Jz_{j_1}\cdots z_{j_k},
    \]
    where $a_{\alpha}\in \mathbb{C}$ and $|\alpha|=\alpha_1+\cdots+\alpha_n$. Given $\alpha$ we have $a_{\alpha}=c_J$ where $J=(1,\overset{\alpha_1}{...},1,...,n,\overset{\alpha_k}{...},n)$.\\
    We will write $z_1^{\alpha_1}\cdots z_n^{\alpha_n}=z^{\alpha}$ and $z_{j_1}\cdots z_{j_k}=z_J$.

    For $1\leq q\leq \infty$ we denote by $\mathcal{P}(^k\ell_q^n)$ the Banach space of all $k-$homogeneous polynomials in $n$ variables with the norm
    \[
        \|p\|_{\mathcal{P}(^k\ell_q^n)}=\sup\{|p(z_1,...,z_n)|:\|(z_1,...,z_n)\|_q\leq 1\}.
    \]

    It is well known that for every $k-$homogeneous polynomial there is a unique symmetric $k-$linear form $L$ on $\mathbb{C}^n$ such that $p(z)=L(z,...,z)$ for all $z\in\mathbb{C}^n$ \cite[Chapter 1]{Di}. Also for each $1\leq q\leq \infty$ and $k\geq 2$ there exists a constant $\lambda(k,q)>0$ such that
    \begin{align}\label{firstineq}
        \|p\|_{\mathcal{P}(^k\ell_q^n)} & \leq \|L\|_{(^k\ell_q^n)}:=\sup\{L(z^{(1)},...,z^{(k)}):\|z^{(j)}\|_q\leq 1, j=1,...,k\}\\
        &\leq \lambda(k,q)\|p\|_{\mathcal{P}(^k\ell_q^n)}.
    \end{align}
    In general $\lambda(k,q)\leq \frac{k^k}{k!}$, but it is worth mentioning that some improvements are known for particular cases: $\lambda(k,2)=1$ and $\lambda(k,\infty)\leq \frac{k^{\frac{k}{2}}(k+1)^{\frac{k+1}{2}}}{2^kk!}$ (see \cite[Propositions 1.43 and 1.44]{Di}).

    We will need to use the following combinatorial object. Let $n\in\mathbb{N}$ and $1\leq t \leq k\leq n$. An $S_p(t,k,n)$ \emph{partial Steiner system} is a collection of subsets of $\{1,...,n\}$ of cardinality $k$ (called blocks) such that every subset of $\{1,...,n\}$ of size $t$ is contained in at most one block of the system.\\
    Useful $k-$homogeneous polynomials can be created from some partial Steiner systems. A $k-$homogeneous polynomial of $n$ variables is a \emph{Steiner unimodular polynomial} if there exists an $S_p(t,k,n)$ partial Steiner system $\mathcal{J}$ such that
    \[
        p(z_1,...,z_n)=\sum_{J\in\mathcal{J}}c_Jz_J \hspace{1cm}\text{and}\hspace{1cm} c_J=\pm1.
    \]

    A Young function $\psi$ is an increasing convex function defined on $(0,\infty)$ such that $\lim_{t\to\infty}\psi(t)=\infty$ and $\psi(0)=0$. For a probability space $(\Omega,\Sigma,\mathbb{P})$, the Orlicz space $L_{\psi}=L_{\psi}(\Omega,\Sigma,\mathbb{P})$ is defined as the space of all real-valued random variables $Z$ for which there exists $c>0$ such that $\mathbb{E}(\psi(|Z|/c))<\infty$. The notion of Orlicz space extends the usual notion of the space $L_p$ with $p\geq 1$.\\
    An Orlicz space is a Banach space with the norm $\|Z\|_{L_{\psi}}=\inf\{c>0:\mathbb{E}(\psi(|Z|/c))\leq 1\}$. For more information about these spaces and random processes see \cite{LeTa}.

    For a metric space $(X,d)$, given $\varepsilon>0$, we denote by $N(X,d;\varepsilon)$, the smallest number of open balls of radius $\varepsilon$ in the metric $d$, which form a covering of the metric space $X$. We also denote by $\diam X$, the diameter of the metric space $X$, that is, the supremum of distances between the points in $X$. \\
    Random processes can be studied by considering an ``entropy condition" relative to the index set equipped with a metric to the random process. Specifically, we need something called the {\em entropy integral}.

    The entropy integral of $(X,d)$ with respect to $\psi$ (Young function) is given by
    \[
        J_{\psi}(X,d):=\int_{0}^{\diam X}\psi^{-1}(N(X,d;\varepsilon))\,d\varepsilon.
    \]

    The next theorem due to Pisier (\cite{Pi}) bounds the expectation of a random process with the entropy integral, provided the process satisfies a contraction condition.
    \begin{theorem}\label{Pisier}
        Let $Z=(Z_x)_{x\in X}$ be a random process indexed by $(X,d)$ in $L_{\psi}$ such that, for every $x,x'\in X$,
        \[
            \|Z_x-Z_{x'}\|_{L_{\psi}}\leq d(x,x').
        \]
        Then if $J_{\psi}(X,d)$ is finite, $Z$ is almost surely bounded and
        \[
            \mathbb{E}\left(\sup_{x,x'\in X}|Z_x-Z_{x'}|\right)\leq 8J_{\psi}(X,d).
        \]
    \end{theorem}

    \section{Some Asymptotic Estimates}

    In \cite{MaTo}, the authors considered the following ``von Neumann's inequality type problem'': for each $1\leq q<\infty$, determine $C_{k,q}(n)$ be the smallest constant such that
    \[
        \|p(T_1,...,T_n)\|_{\mathcal{L}(\mathcal{H})}\leq C_{k,q}(n)\sup\{|p(z_1,...,z_n)|:\sum_{j=1}^n|z_j|^q\leq 1\},
    \]
    for every $k-$homogeneous polynomial $p$ in $n$ variables and every $n-$tuple of commuting contractions $(T_1,...,T_n)$ with $\sum_{j=1}^n\|T_i\|_{\mathcal{\mathcal{L}(\mathcal{H})}}^q\leq 1$. They obtained lower and upper estimates for the growth of $C_{k,q}(n)$
    \cite[Propositions 11 and 17]{MaTo} (here $q'$ denotes the conjugate of $q$):
    \begin{align}
        n^{\frac{k-1}{q'}-\frac{1}{2}\left[\frac{k}{2}\right]}\ll C_{k,q}(n)&\ll n^{\frac{k-2}{q'}} \hspace{0.2cm}\text{for}\hspace{0.2cm} 1\leq q\leq 2,\\
        n^{\frac{k}{2}-\frac{1}{2}\left(\left[\frac{k}{2}\right]+1\right)}\ll C_{k,q}(n)&\ll n^{\frac{k-2}{2}} \hspace{0.2cm}\text{for}\hspace{0.2cm} 2\leq q< \infty.\label{upperbound}
    \end{align}
    Furthermore, the upper bounds here hold for every $n-$tuple $(T_1,...,T_n)$ satisfying $\sum_{j=1}^n\|T_i\|_{\mathcal{\mathcal{L}(\mathcal{H})}}^q\leq 1$ (and even a weaker condition), not necessarily commuting. If we do not ask the contractions to commute, this bound is shown to be optimal in \cite[Proposition 15]{MaTo}.

    In \cite{Ga}, the lower estimates were improved and the asymptotic behavior of $C_{k,\infty}(n)$ was established. Specifically, Galicer et al. showed:
    \begin{theorem}[Galicer D., Muro S., Sevilla-Peris P.]\label{maintheo}
        For $k\geq 3$ and $1\leq q\leq\infty$, let $C_{k,q}(n)$ be the smallest constant such that
        \[
            \|p(T_1,...,T_n)\|_{\mathcal{L}(\mathcal{H})}\leq C_{k,q}(n)\sup\{|p(z_1,...,z_n)|:\|(z_j)_j\|_q\leq 1\},
        \]
        for every $k-$homogeneous polynomial $p$ in $n$ variables and every $n-$tuple of commuting contractions $(T_1,...,T_n)$ with $\sum_{j=1}^n\|T_i\|_{\mathcal{\mathcal{L}(\mathcal{H})}}^q\leq 1$. Then
        \begin{enumerate}
            \item[(i)] $C_{k,\infty}(n)\sim n^{\frac{k-2}{2}}$
            \item[(ii)] for $2\leq q<\infty$ we have
            \[
                \log^{-3/q}(n)n^{\frac{k-2}{2}}\ll C_{k,q}(n) \ll n^{\frac{k-2}{2}}.
            \]
            In particular, $n^{\frac{k-2}{2}-\varepsilon} \ll C_{k,q}(n) \ll n^{\frac{k-2}{2}}$ for every $\varepsilon >0$.
        \end{enumerate}
    \end{theorem}

    In \cite[Theorem 2.5]{Ga} the authors proved the following:
    \begin{theorem}[Galicer D., Muro S., Sevilla-Peris P.]\label{secondtheorem}
        Let $k\geq 2$ and $\mathcal{J}$ be an $S_p(k-1,k,n)$ partial Steiner system. Then there exist signs $(c_J)_{J\in\mathcal{J}}$ and a constant $A_{k,q}>0$ independent of $n$ such that the $k-$homogeneous polynomial $p=\sum_{J\in\mathcal{J}}c_Jz_J$ satisfies
        \[
            \|p\|_{\mathcal{P}(^k\ell^n_q)}\leq A_{k,q}\times\begin{cases}
                \log^{\frac{3}{q}}(n)n^{\frac{k}{2}(\frac{q-2}{q})} \text{ for } 2\leq q <\infty\\
                \log^{\frac{3q-3}{q}}(n) \hspace{0.75cm}\text{   for } 1\leq q\leq 2.
                \end{cases}
        \]
        Moreover, the constant $A_{k,q}$ may be taken independent of $k$ for $q\neq 2$.
    \end{theorem}

    Galicer et al. also mentioned the next remark:
    \begin{remark}
        It is not difficult to prove that every 2-homogeneous Steiner unimodular polynomial has norm in $\mathcal{P}(^2\ell^n_2)$ less than or equal to $\frac{1}{2}$. It would be interesting to know if there is a constant $C$, perhaps depending on $k\geq 3$ and not on $n$, such that given any $S_p(k-1,k,n)$ partial Steiner system $\mathcal{J}$, we can find a $k-$homogeneous unimodular polynomial $p(z):=\sum_{J\in\mathcal{J}}c_Jz_J$ with $\|p\|_{\mathcal{P}(^k\ell^n_2)}\leq C$. An affirmative answer to this question would in particular give that the upper bound given by Mantero and Tonge (\ref{upperbound}) for $C_{k,q}(n)$ with $2\leq q<\infty$ is actually optimal.
    \end{remark}

    Let's expand this remark a little bit. First, why does every 2-homogeneous Steiner unimodular polynomial have norm in $\mathcal{P}(^2\ell^n_2)$ less than or equal to $\frac{1}{2}$? Consider $S_p(1,2,n)$, let $r:=n$ if $n$ is even and $r:=n-1$ if $n$ is odd. Then the partial Steiner system is $\{\{1,2\},\{3,4\},...,\{r-1,r\}\}$ which gives rise to the polynomial $p(z)=z_1z_2+z_3z_4+...+z_{r-1}z_r$. In this way
    \begin{align*}
        |p(z)|&\leq\frac{1}{2}\left(|z_1z_2|+|z_3z_4|...+|z_{r-1}z_r|+|z_2z_1|+|z_4z_3|...+|z_rz_{r-1}|\right)\\
        &\leq \frac{1}{2}(|z_1|^2+...|z_r|^2)^{1/2}(|z_1|^2+...|z_r|^2)^{1/2}.
    \end{align*}
    Thus $\|p\|_{\mathcal{P}(^2\ell^n_2)}\leq\frac{1}{2}$

    Second, let's use the norm in $\mathcal{P}(^k\ell^n_2)$. Consider $S_p(k-1,k,k+1)$. Then one representative for the partial Steiner system is $\{\{1,2,...,k\}\}$ which gives rise to the polynomial $p(z)=z_1z_2\cdots z_k$. By a generalized H\"{o}lder's Inequality, if $|z_1|^2+|z_2|^2+\cdots+|z_k|^2\leq 1$, then
    \begin{align*}
        |p(z)|&\leq\frac{1}{k}\left(|z_1z_2\cdots z_k|+|z_2\cdots z_kz_1|+\cdots+|z_kz_1\cdots z_{k-1} |\right)\\
        &\leq \frac{1}{k}\left(|z_1|^k+|z_2|^k+\cdots+|z_k|^k\right)^{1/k}\left(|z_2|^k+|z_3|^k+\cdots+|z_1|^k\right)^{1/k}\cdots\\
        &\phantom{=}\cdots\left(|z_k|^k+|z_1|^k+\cdots+|z_{k-1}|^k\right)^{1/k}\\
        &\leq\frac{1}{k}\left(|z_1|^2+|z_2|^2+\cdots+|z_k|^2\right)^{1/2}\left(|z_2|^2+|z_3|^2+\cdots+|z_1|^2\right)^{1/2}\cdots\\
        &\phantom{=}\cdots\left(|z_k|^2+|z_1|^2+\cdots+|z_{k-1}|^2\right)^{1/2}\\
        &\leq\frac{1}{k}.
    \end{align*}
    So, at least for this particular partial Steiner system $S_p(k-1,k,k+1)$ we can conclude that $\|p\|_{\mathcal{P}(^k\ell^n_2)}\leq\frac{1}{k}$.

    Galicer et al. also obtained the following corollary, using an observation about the cardinality of partial Steiner systems:\\
    It is well known that any partial Steiner system $S_p(t,k,n)$ has cardinality less than or equal to ${n \choose t}/{k \choose t}$ \cite{Ro}. In \cite{Noga}, it is proved that there exists a constant $c>0$ such that there exist partial Steiner systems $S_p(k-1,k,n)$ of cardinality at least
    \begin{equation}\label{cardsystem}
        \psi(k,n):=\begin{cases}
            \frac{{n \choose k-1}}{k}\left(1-\frac{c}{n^{\frac{1}{k-1}}}\right) \text{ for } k>3,  \\
            \frac{{n \choose k-1}}{k}\left(1-\frac{c\log^{3/2}n}{n^{\frac{1}{k-1}}}\right)  \text{ for } k=3.
        \end{cases}
    \end{equation}
    As a consequence:
    \begin{corollary}[Galicer D., Muro S., Sevilla-Peris P.]\label{cor}
        Let $k\geq 3$. Then there exists a $k-$homogeneous Steiner unimodular polynomial $p$ of $n$ complex variables with at least $\psi(k,n)$ coefficients satisfying the estimates in Theorem \ref{secondtheorem}. Note that in this case $\psi(k,n)>>n^{k-1}$.
    \end{corollary}

    Mantero and Tonge \cite{MaTo} also studied another multivariable extension of von Neumann inequality by considering polynomials in commuting operators $T_1,...,T_n$ satisfying that for any pair $h,g$ of norm one vectors in the Hilbert space
    \begin{equation}\label{cond}
        \sum_{j=1}^{n}|\langle T_jh,g\rangle|^q\leq 1,
    \end{equation}
    or, equivalently, that for any vector $\alpha\in \mathbb{C}^n$ such that $\|\alpha\|_{\ell^n_{q'}}=1$ we have
    \[
        \|\sum_{j=1}^{n}\alpha_jT_j\|\leq 1.
    \]
    Let $D_{k,q}(n)$ be the smallest constant such that
    \[
        \|p(T_1,...,T_n)\|_{\mathcal{L}(\mathcal{H})}\leq D_{k,q}(n)\sup\{|p(z_1,...,z_n)|:\sum_{j=1}^n|z_j|^q\leq 1\},
    \]
    for every $k-$homogeneus polynomial $p$ in $n$ variables and every $n-$tuple of contractions satisfying (\ref{cond}). The upper bound for $D_{k,q}(n)$ found in \cite[Proposition 20]{MaTo} is
    \[
        D_{k,q}(n)\begin{cases} n^{(k-1)(\frac{1}{2}+\frac{1}{q})} , \hspace{0.2cm}\text{for}\hspace{0.2cm} q\geq 2,\\ n^{(k-1)(\frac{1}{2}+\frac{1}{q'})} , \hspace{0.2cm}\text{for}\hspace{0.2cm} q\leq 2.\end{cases}
    \]
    For $k=3$ and $q=2$, Galicer et al. showed that this is an optimal bound up to a logarithmic factor. Specifically they showed:
    \begin{proposition}[Galicer D., Muro S., Sevilla-Peris P.]
        We have the following asymptotic behavior:
        \[
            \frac{n^2}{\log^{15/4}n}\ll D_{3,2}(n)\ll n^2.
        \]
    \end{proposition}

    Mirrowing techniques from \cite{Ga} we can go further:
    \begin{proposition}[Zatarain-Vera O.]\label{mainprop}
        We have the following asymptotic behavior for $D_{k,2}(n)$:
        \[
            \frac{n^{k-1}}{\log^{\frac{3}{4}(k+2)}n}\ll D_{k,2}(n) \ll n^{k-1}.
        \]
    \end{proposition}

    Before presenting the proof of this proposition, let's recall the construction of commuting contractions which appeared in the proof of Theorem \ref{maintheo}
    \cite[Theorem 1.1]{Ga}. Fix $k, n \in\mathbb{N}$ with $n\geq k\geq 3$. Let $\mathcal{J}$ be a partial Steiner system $S_p(k-1,k,n)$ such that $|\mathcal{J}|=\psi(k,n)$ as in (\ref{cardsystem}). Let $\mathcal{H}$ be a finite dimensional Hilbert space with the following orthonormal basis:
    \[
        \begin{cases}
            e;\\
            e(j_1,...,j_m) \hspace{0.3cm} 1\leq m\leq k-2 \hspace{0.3cm}\text{and}\hspace{0.3cm} 1\leq j_1\leq j_2\leq \cdots \leq j_m\leq n;\\
            f_i \hspace{1.75cm}1\leq i\leq n;\\
            g.
        \end{cases}
    \]
    Given any subset $\{i_1,...,i_r\}\subset\{1,...,n\}$ we denote by $[i_1,...,i_r]$ its nondecreasing reordering. For $1\leq l\leq n$, we define the operator $T_l\in B(\mathcal{H})$ as
    \[
        \begin{cases}
            T_le=e(l);\\
            T_le(j_1,...,j_m)=e[l,j_1,...,j_m] \hspace{0.3cm} \text{if} \hspace{0.3cm} 1\leq m <k-2;\\
            T_le(j_1,...,j_{k-2})=\sum_{i}\gamma_{\{i,l,j_1,...,j_{k-2}\}}f_i;\\
            T_lf_i=\delta_{li}g;\\
            T_lg=0,
        \end{cases}
    \]
    where $\gamma$ is defined as
    \[
        \gamma_{\{i_1,...,i_k\}}:=\begin{cases} c_{\{i_1,...,i_k\}} \hspace{0.3cm} \text{if} \hspace{0.1cm} \{i_1,...,i_k\}\in \mathcal{J} \\ 0 \hspace{1.5cm}\text{otherwise} \end{cases}
    \]
    Since $\mathcal{J}$ is a partial Steiner system $S_p(k-1,k,n)$, it has been proved that $(T_1,...,T_n)$ is a commuting tuple of contractions on $\mathcal{H}$ and that $\|T_l\|=1$ for $l=1,...,n$. Following \cite{Ra}, we call such a commuting tuple of contractions a Dixon's $n-$tuple.

    \begin{proof}[Proof of Proposition \ref{mainprop}]
        Let $p(z)=\sum_{J\in \mathcal{J}}c_Jz_J$ be a $k-$homogeneous Steiner unimodular polynomial as in Theorem \ref{secondtheorem}. Let $(T_1,...,T_n)$ be a Dixon's $n-$tuple and $h\in \mathcal{H}$. First, we prove that
        \[
            \displaystyle\frac{T_1}{(1+\|p\|_{\mathcal{P}(^k\ell_2^n)})^{1/2}},...,\displaystyle\frac{T_n}{(1+\|p\|_{\mathcal{P}(^k\ell_2^n)})^{1/2}}
        \]
        satisfy (\ref{cond}) to get the upper bound.

        We have for $\alpha \in \ell_2^n$ and $1\leq m \leq k-2$:
        \begin{align*}
            \sum_j\alpha_jT_jh&=\sum_j\alpha_j\left[\langle h,e \rangle T_je
            +\sum_{j_1}\cdots\sum_{j_m}\langle h,e(j_1,...,j_m) \rangle T_je(j_1,...,j_m)\right.
            \\
            &\phantom{=}\left.
            +\sum_{i}\langle h,f_i \rangle T_jf_i
            +\langle h,g\rangle T_jg \right]\\
            \\
            &=\sum_j\alpha_j\langle h,e \rangle T_je\\
            &+\sum_j\sum_{j_1}\cdots\sum_{j_m}\alpha_j\langle h,e(j_1,...,j_m) \rangle T_je(j_1,...,j_m)\\
            &+\sum_j\sum_{j_{1}}\cdots\sum_{j_{k-2}}\alpha_j\langle h,e(j_1,...,j_{k-2}) \rangle T_je(j_1,...,j_{k-2})\\
            &+\sum_j\sum_{i}\alpha_j\langle h,f_i \rangle T_jf_i\\
            &+\sum_j\alpha_j\langle h,g\rangle T_jg\\
            \\
            &=\sum_j\alpha_j\langle h,e \rangle e(j)\\
            &+\sum_j\sum_{j_1}\cdots\sum_{j_m}\alpha_j\langle h,e(j_1,...,j_m) \rangle e[j,j_1,...,j_m]\\
            &+\sum_j\sum_{j_1}\cdots\sum_{j_{k-2}}\alpha_j\langle h,e(j_1,...,j_{k-2}) \rangle \sum_i \gamma_{\{i,j,j_1,...,j_{k-2}\}}f_i\\
            &+\sum_j\sum_{i}\alpha_j\langle h,f_i \rangle \delta_{ji}g\\
        \end{align*}
        \begin{align*}
            &=\sum_j\alpha_j\langle h,e \rangle e(j)\\
            &+\sum_j\sum_{j_1}\cdots\sum_{j_m}\alpha_j\langle h,e(j_1,...,j_m) \rangle e[j,j_1,...,j_m]\\
            &+\sum_j\sum_i\sum_{j_1}\cdots\sum_{j_{k-2}}\alpha_j\langle h,e(j_1,...,j_{k-2}) \rangle c_{\{i,j,j_1,...,j_{k-2}\}}f_i\\
            &+\sum_j\alpha_j\langle h,f_j \rangle g.
        \end{align*}
        In this way (below $\beta$ is an appropriate vector in the unit ball of $\ell_2^n$)
        \begin{align*}
            \|\sum_j\alpha_jT_jh\|^2&=\sum_j|\alpha_j\langle h,e \rangle|^2\\
            &+\sum_j\sum_{j_1,...,j_m}|\alpha_j\langle h,e(j_1,...,j_m) \rangle|^2\\
            &+\sum_j\left|\sum_{l,j_1,...,j_{k-2}}\alpha_j\langle h,e(j_1,...,j_{k-2}) \rangle c_{\{j,l,j_1,...,j_{k-2}\}}\right|^2\\
            &+\left|\sum_j\alpha_j \langle h,f_j\rangle\right|^2\\
            \\
            &\leq \|\alpha\|^2_{\ell_2}|\langle h,e \rangle|^2 + \|\alpha\|^2_{\ell_2}\|(\langle h,e(j_1,...,j_m) \rangle)_{j_1,...,j_m}\|^2_{\ell_2}\\
            &+\left(\sum_j\beta_j \sum_{l,j_1,...,j_{k-2}}\alpha_j\langle h,e(j_1,...,j_{k-2}) \rangle c_{\{j,l,j_1,...,j_{k-2}\}}\right)^2\\
            &+\|\alpha\|^2_{\ell_2}\| \|(\langle h,f_j \rangle)_j\|^2_{\ell_2}\\
            \\
            &\leq \|\alpha\|^2_{\ell_2}|\langle h,e \rangle|^2 + \|\alpha\|^2_{\ell_2}\|(\langle h,e(j_1,...,j_m) \rangle)_{j_1,...,j_m}\|^2_{\ell_2}\\
            &+\|\alpha\|^2_{\ell_2} \|p\|_{\mathcal{P}(^k\ell_2^n)} \|(\langle h,e(j_1,...,j_{k-2}) \rangle)_{j_1,...,j_{k-2}}\|^2_{\ell_2}\\
            &+\|\alpha\|^2_{\ell_2} \|(\langle h,f_j \rangle)_j\|^2_{\ell_2}\\
            \\
            &=\|\alpha\|^2_{\ell_2}(|\langle h,e \rangle|^2 + \|(\langle h,e(j_1,...,j_m) \rangle)_{j_1,...,j_m}\|^2_{\ell_2}
            +\|(\langle h,f_j \rangle)_j\|^2_{\ell_2^n}\\
            &+ \|p\|_{\mathcal{P}(^k\ell_2^n)} \|(\langle h,e(j_1,...,j_{k-2}) \rangle)_{j_1,...,j_{k-2}}\|^2_{\ell_2})\\
            \\
            &\leq \|\alpha\|^2_{\ell_2} \left(\|h\|^2_{\mathcal{H}}+\|p\|_{\mathcal{P}(^k\ell_2^n)}\|h\|^2_{\mathcal{H}}\right)\\
            &=\|\alpha\|^2_{\ell_2} \|h\|^2_{\mathcal{H}}\left(1+\|p\|_{\mathcal{P}(^k\ell_2^n)}\right).
        \end{align*}.
        Thus $\|\sum_j \alpha_j T_j\|\leq \|\alpha\|_{\ell_2} \left(1+\|p\|_{\mathcal{P}(^k\ell_2^n)}\right)^{1/2}$ and so
        \[
            \displaystyle\frac{T_1}{(1+\|p\|_{\mathcal{P}(^k\ell_2^n)})^{1/2}},...,\displaystyle\frac{T_n}{(1+\|p\|_{\mathcal{P}(^k\ell_2^n)})^{1/2}}
        \]
        satisfy (\ref{cond}).

        For the lower bound on $D_{k,2}(n)$, observe that
        \begin{align*}
            p(T_1,...,T_n)e &= \sum_{\{i_1,...,i_k\}\in\mathcal{J}}c_{\{i_1,...,i_k\}}T_{i_1}T_{i_2}...T_{i_k}e\\
            &= \sum_{\{i_1,...,i_k\}\in\mathcal{J}}c^2_{\{i_1,...,i_k\}}g = |\mathcal{J}|g.
        \end{align*}
        Thus $\|p(T_1,...,T_n)e\|_{\mathcal{H}}=|\mathcal{J}|$.

        We use the above observation together with Theorem \ref{secondtheorem} and Corollary \ref{cor} (recall the corollary gives a lower bound for the cardinality of the partial Steiner system in question) to get:
        \begin{align*}
            \left\|p\left(\frac{T_1}{(1+\|p\|_{\mathcal{P}(^k\ell_2^n)})^{1/2}},...,\frac{T_n}{(1+\|p\|_{\mathcal{P}(^k\ell_2^n)})^{1/2}}\right)\right\|_{\mathcal{L}(\mathcal{H})}\geq
        \end{align*}
        \begin{align*}
            &\geq  (1+\|p\|_{\mathcal{P}(^k\ell_2^n)})^{-k/2} \|p(T_1,...,T_n)e\|_{\mathcal{H}}\\
            &= (1+\|p\|_{\mathcal{P}(^k\ell_2^n)})^{-k/2} |\mathcal{J}|\\
            &= \frac{1}{(1+\|p\|_{\mathcal{P}(^k\ell_2^n)})^{k/2}} \frac{\|p\|_{\mathcal{P}(^k\ell_2^n)}}{\|p\|_{\mathcal{P}(^k\ell_2^n)}} |\mathcal{J}|\\
            &\geq \frac{1}{C'(\log^{3/2}n)^{k/2}} \frac{\|p\|_{\mathcal{P}(^k\ell_2^n)}}{A_{k,2}\log^{3/2}n} n^{k-1}\\
            &= C \frac{\|p\|_{\mathcal{P}(^k\ell^n_2)}}{(\log^{3/2}n)^{\frac{k}{2}+1}} n^{k-1}\\
            &= C \frac{\|p\|_{\mathcal{P}(^k\ell^n_2)}}{\log^{\frac{3}{4}(k+2)}n} n^{k-1}\\
            &\gg \|p\|_{\mathcal{P}(^k\ell^n_2)}\frac{n^{k-1}}{(\log n)^{\frac{3}{4}(k+2)}}.
        \end{align*}
        We used that for large enough $n$ we have $1\leq \log^{3/2}n$, and by Theorem \ref{secondtheorem} also $\|p\|_{\mathcal{P}(^k\ell^n_2)}\leq A_{k,2}\log^{3/2}n$, and consequently $(1+\|p\|_{\mathcal{P}(^k\ell^n_2)})\leq C'\log^{3/2}n$, where $C'$ is a constant independent of $n$.
    \end{proof}

    Our next objective is to prove a variant of Theorem \ref{secondtheorem}.

    Let $n\geq k\geq 3$ and let $\mathcal{J}$ be a $S_p(k-1,k,n)$ partial Steiner system. Consider a family of independent Bernoulli variables $(\varepsilon_J)_{J\in\mathcal{J}}$ on $(\Omega,\Sigma,\mathbb{P})$. For $z\in B_{\ell_{2}^{n}}$ we define the following Rademacher process indexed by $B_{\ell_{2}^{n}}$ as
    \begin{equation}\label{process}
        Y_z=\frac{1}{k}\sum_{J\in\mathcal{J}} \varepsilon_Ja_Jz_J,
    \end{equation}
    where the $(a_J)_{J\in\mathcal{J}}$ are complex coefficients.
    We view it as a random process in the Orlicz space defined by the Young function $\psi_2(t)=e^{t^2}-1$. Recall $(\varepsilon_i)$ still spans a subspace isomorphic to $\ell_2$ in $\psi_2$.

    \begin{lemma}\label{Lipcond}
        The Rademacher process defined in (\ref{process}) fulfills the following Lipschitz condition
        \[
            \|Y_z-Y_{z'}\|_{L_{\psi_2}}\leq C\left(\max_{J\in\mathcal{J}}|a_J|\right)\|z-z'\|_{\infty}
        \]
        for some universal constant $C\geq 1$ and every $z,z'\in B_{\ell_{2}^{n}}$.
    \end{lemma}
    \begin{proof}
        By Khintchine's inequality the $\psi_2-$norm of a Rademacher process is comparable to its $L_2-$norm. So,
        \begin{align*}
            \|Y_z-Y_{z'}\|_{L_2}&=\displaystyle\frac{1}{k}\left(\int_{\Omega}\left|\displaystyle\sum_{J\in\mathcal{J}}\varepsilon_J(\omega)a_J(z_J-z'_{J})\right|^2\,d\mathbb{P}(\omega)\right)^{1/2}
            =\frac{1}{k}\left(\displaystyle\sum_{J\in\mathcal{J}}|a_J|^2|z_J-z'_{J}|^2\right)^{1/2}\\
            &=\displaystyle\frac{1}{k}\left(\displaystyle\sum_{J\in\mathcal{J}}|a_J|^2\left|\sum_{u=1}^{k}z_{j_1}\cdots z_{j_{u-1}}(z_{j_u}-z'_{j_u})z'_{j_{u+1}}\cdots z'_{j_k}\right|^2\right)^{1/2}\\
            &\leq \displaystyle\frac{1}{k}\sum_{u=1}^{k}\left(\displaystyle\sum_{J\in\mathcal{J}}|a_J|^2|z_{j_1}\cdots z_{j_{u-1}}(z_{j_u}-z'_{j_u})z'_{j_{u+1}}\cdots z'_{j_k}|^2\right)^{1/2}\\
            &\leq \displaystyle\frac{1}{k}\sum_{u=1}^{k}\|z-z'\|_{\infty}\left(\displaystyle\sum_{J\in\mathcal{J}}|a_J|^2|z_{j_1}\cdots z_{j_{u-1}}z'_{j_{u+1}}\cdots z'_{j_k}|^2\right)^{1/2}\\
            &\leq \displaystyle\frac{1}{k}\sum_{u=1}^{k}\|z-z'\|_{\infty}\left(\max_{J\in\mathcal{J}}|a_J|^2\right)^{1/2}\left(\displaystyle\sum_{J\in\mathcal{J}}|z_{j_1}\cdots z_{j_{u-1}}z'_{j_{u+1}}\cdots z'_{j_k}|^2\right)^{1/2}\\
            &\leq\left(\displaystyle\max_{J\in\mathcal{J}}|a_J|\right)\|z-z'\|_{\infty},
        \end{align*}
        where the last inequality is true for the following reason: Since $\mathcal{J}$ is an $S_p(k-1,k,n)$ partial Steiner system, given $z_{j_1}\cdots z_{j_{u-1}}z'_{j_{u+1}}\cdots z'_{j_k}$ for a fixed $u$, there is at most one index $j_u$ such that $(j_1,...,j_k)$ belongs to $\mathcal{J}$. Therefore the sum $\sum_{J\in\mathcal{J}}|z_{j_1}\cdots z_{j_{u-1}}z'_{j_{u+1}}\cdots z'_{j_k}|^2$ can be bounded by
        \[
            (\sum_{l_1=1}^{n}|z_{l_1}|^2)\cdots(\sum_{l_{u-1}=1}^{n}|z_{l_{u-1}}|^2)(\sum_{l_{u+1}=1}^{n}|z'_{l_{u+1}}|^2)\cdots(\sum_{l_k=1}^{n}|z'_{l_k}|^2),
        \]
        which is less or equal than one since $z,z'$ are in the unit ball of $\ell_2^n$.
    \end{proof}

    So, by Lemma \ref{Lipcond} and Pisier's theorem with $L_{\psi_2}, X=B_{\ell_2^n}$ and $d=\|\cdot\|_{\infty}$, to bound the expectation of the supremum of the random process, we need to estimate the integral $J_{\psi_2}(B_{\ell_2^n},\|\cdot\|_{\infty})$. Note that $\psi_{2}^{-1}(t)=\log^{1/2}(t+1)$ but we can use instead $\log^{1/2}(t)$ since it does not change the computation in the integral.

    \begin{lemma}\label{entropy}
        There exists $C>0$ such that for every $n\geq 2$ we have $J_{\psi_2}(B_{\ell_2^n},\|\cdot\|_{\infty})\leq C\log^{3/2}(n)$.
    \end{lemma}
    \begin{proof}
        See \cite{Ga}.
    \end{proof}

    Now we can adapt the proof of \cite[Theorem 2.5]{Ga} to obtain the next result:
    \begin{theorem}[Zatarain-Vera O.]
        Let $n\geq k\geq 3$ and $\mathcal{J}$ be an $S_p(k-1,k,n)$ partial Steiner system. Then there exist signs $(c_J)_{J\in\mathcal{J}}$, complex coefficients $(a_J)_{J\in\mathcal{J}}$, and a constant $A_{k,q}>0$ independent of $n$, such that the $k-$homogeneous polynomial $p=\sum_{J\in\mathcal{J}}c_Ja_Jz_J$ satisfies
        \[
            \|p\|_{\mathcal{P}(^k\ell^n_q)}\leq A_{k,q}\times\begin{cases}
                \|a_J\|_{\infty} \log^{\frac{3}{q}}(n)n^{\frac{k}{2}\frac{(q-2)}{q}}, \text{ for } 2\leq q <\infty\\
                \|a_J\|_{\infty} \log^{\frac{3q-3}{q}}(n), \text{   for } 1\leq q\leq 2.
                \end{cases}
        \]
        Moreover, the constant $A_{k,q}$ may be taken independent of $k$ for $q\neq 2$.
    \end{theorem}
    \begin{proof}
        Any $S_p(k-1,k,n)$ partial Steiner system $\mathcal{J}$ satisfies $|\mathcal{J}|\leq\frac{1}{k}{n \choose k-1}$, now we use $\mathcal{J}$ to define a Rademacher process $(Y_z)_{z\in B_{\ell^n_2}}$ as in (\ref{process}). By Lemmas \ref{Lipcond}, \ref{entropy} and Theorem \ref{Pisier} there exists a constant $K>0$ such that $\mathbb{E}(\sup_{z\in B_{\ell_{2}^{n}}}|Y_z|)\leq K\|a_J\|_{\infty}\log^{3/2}(n)$. By Markov's inequality we have
        \[
            \mathbb{P}\{\omega\in\Omega:\|\sum_{J\in\mathcal{J}}\varepsilon_J(\omega)a_Jz_J\|_{\mathcal{P}(^k\ell_2^n)}\geq MkK\|a_J\|_{\infty}\log^{3/2}(n)\}\leq \frac{1}{M},
        \]
        where $M$ is a constant to be determined. Now, recall that by \cite{Ka} we have
        \[
            \mathbb{P}\{\omega\in\Omega:\|\sum_{J\in\mathcal{J}}\varepsilon_J(\omega)a_Jz_J\|_{\mathcal{P}(^k\ell_{\infty}^n)}\geq D(n|\mathcal{J}|\|a_J\|^2_{\infty}\log(k))^{1/2}\}\leq \frac{1}{k^2e^n}.
        \]
        Therefore if $M>1+\frac{1}{k^2e^n-1}$ then $\frac{1}{M}+\frac{1}{k^2e^{n}}<1$, so for $\omega$ in a set of positive measure, we have the following
        \begin{equation}\label{ineq}\begin{cases}
            \|\sum_{J\in\mathcal{J}}\varepsilon_J(\omega)a_Jz_J\|_{\mathcal{P}(^k\ell_2^n)}\leq MkK\|a_J\|_{\infty}\log^{3/2}(n),\\
            \|\sum_{J\in\mathcal{J}}\varepsilon_J(\omega)a_Jz_J\|_{\mathcal{P}(^k\ell_{\infty}^n)}
            \leq D(\frac{\log(k)}{k}{n\choose k-1}n\|a_J\|^2_{\infty})^{1/2}
            \leq D(\frac{\log(k)}{k!}n^k\|a_J\|^2_{\infty})^{1/2}.
        \end{cases}\end{equation}
        Then, there is a choice of signs and coefficients $(a_J)_{J\in\mathcal{J}}$ such that $p(z):=\sum_{J\in{\mathcal{J}}}c_ja_Jz_J$ satisfies the inequalities in (\ref{ineq}). We shall use an interpolation argument (see \cite{BeLo}) to get a bound for the norm of $p$ in $\mathcal{P}(^k\ell_2^n)$ for $2<q<\infty$. Consider the $k-$linear form $L$ associated to $p$. Then by interpolation, and inequalities (\ref{firstineq}) and  (\ref{ineq}) we obtain
        \begin{align*}
            \|p\|_{\mathcal{P}(^k\ell_q^n)} &\leq \|L\|_{(^k\ell_q^n)} \leq \|L\|^{\frac{2}{q}}_{(^k\ell_2^n)}\|L\|^{\frac{q-2}{q}}_{(^k\ell_{\infty}^n)}\\
            &\leq \lambda(k,2)^{\frac{2}{q}}\|p\|^{\frac{2}{q}}_{\mathcal{P}(^k\ell_2^n)} \lambda(k,\infty)^{\frac{q-2}{q}}\|p\|^{\frac{q-2}{q}}_{\mathcal{P}(^k\ell_{\infty}^n)}\\
            &\leq \left(MkK\|a_J\|_{\infty}\right)^{\frac{2}{q}}\log^{\frac{3}{q}}(n)\left(D\lambda(k,\infty)\frac{\log^{\frac{1}{2}}(k)}{\sqrt{k!}}\|a_J\|_{\infty}\right)^{\frac{q-2}{q}}n^{\frac{k}{2}(\frac{q-2}{q})}\\
            &\leq
            \max\{M,K,D\}\left(\frac{k^{\frac{k}{2}}(k+1)^{\frac{k+1}{k}}\sqrt{\log(k)}}{2^kk!\sqrt{k!}}\right)^{\frac{q-2}{q}}k^{\frac{2}{q}}
            \|a_J\|^{\frac{2}{q}}_{\infty}\|a_J\|^{\frac{q-2}{q}}_{\infty} \log^{\frac{3}{q}}(n)n^{\frac{k}{2}(\frac{q-2}{q})}\\
            &\leq \underbrace{\max\{M,K,D\}\left(\frac{k^{\frac{k}{2}}(k+1)^{\frac{k+1}{k}}\sqrt{\log(k)}}{2^kk!\sqrt{k!}}\right)^{\frac{q-2}{q}}k^{\frac{2}{q}}}_{A_{k,q}}\|a_J\|_{\infty} \log^{\frac{3}{q}}(n)n^{\frac{k}{2}(\frac{q-2}{q})}.
        \end{align*}
        Note that for $q>2, A_{q,k}\to 0$ as $k\to\infty$ and thereby we can take a constant independent of $k$ in this case.\\
        For $q=1$, let $P(z)=\sum_{|\alpha|=k}a_{\alpha}z^{\alpha}$ be any $k-$homogeneous polynomial, then
        \[
            |P(z)|\leq \sum_{|\alpha|=k}|a_{\alpha}z^{\alpha}|\leq \sup_{|\alpha|=k}\left\{|a_{\alpha}\frac{\alpha!}{k!}|\right\}\sum_{|\alpha|=k}|\frac{k!}{\alpha!}z^{\alpha}|=
            \sup_{|\alpha|=k}\left\{|a_{\alpha}\frac{\alpha!}{k!}|\right\}\left(\sum_{j=1}^{n}|z_j|\right)^k.
        \]
        In particular the polynomial $p$ considered above satisfies $\|p\|_{\mathcal{P}(^k\ell_1^n)}\leq\frac{\|a_J\|_{\infty}}{k!}$. So, finally proceeding again by interpolation between the $\ell_1^n$ and $\ell_2^n$ cases we have for $1<q<2$ the next estimate
        \begin{align*}
            \|p\|_{\mathcal{P}(^k\ell_q^n)} &\leq \|L\|_{(^k\ell_q^n)} \leq \|L\|^{\frac{2-q}{q}}_{(^k\ell_1^n)}\|L\|^{\frac{2q-2}{q}}_{(^k\ell_2^n)}\\
            &\leq \lambda(k,1)^{\frac{2-q}{q}}\|p\|^{\frac{2-q}{q}}_{\mathcal{P}(^k\ell_1^n)} \lambda(k,2)^{\frac{2q-2}{q}}\|p\|^{\frac{2q-2}{q}}_{\mathcal{P}(^k\ell_2^n)}\\
            &\leq \left(\|a_J\|_{\infty}\frac{k^k}{(k!)^2}\right)^{\frac{2-q}{q}}(MkK\|a_J\|_{\infty}\log^{3/2}(n))^{\frac{2q-2}{q}}\\
            &=A_{k,q}\|a_J\|_{\infty} \log^{\frac{3q-3}{q}}(n).
        \end{align*}
        Also, in this case, for every $1\leq q<2$ we have $A_{k,q}\to 0$ as $k\to \infty$.
    \end{proof}

\end{document}